\documentclass[11pt,reqno]{amsart}
\usepackage{url}
\usepackage{parskip}
\usepackage{mathtools,amsthm,amssymb}  
\usepackage[english]{babel}
\usepackage[utf8]{inputenc}
\usepackage[T1]{fontenc}

\calclayout

\newcommand{\N}{{\mathbb N}}

\newcommand{\restring}[1]{\mathbb{Z}_{#1}}

\theoremstyle{plain}
\numberwithin{equation}{section}
\newtheorem{thm}{Theorem}[section]

\newtheorem{lemma}[thm]{Lemma}

\newtheorem{corollary}[thm]{Corollary}

\begin{document}

\title{A consideration of the Fibonacci sequence modulo m}
\author{Daniel Ambrée}
\thanks{Daniel Ambrée is student of the \glqq Freie Universität Berlin\grqq \ (Germany), \textit{Email address:} danjj1878@gmail.com}

\begin{abstract}
In this paper, we study natural numbers $m$ with $u_{z \pm 1} \equiv -1 \ (m)$ for a $z \in \N$, where $u_n$ is the $n$th Fibonacci number. Furthermore, we want to show for $m \in \N\setminus\{0,1\}$:
\begin{displaymath}
[\forall p \ \text{prim}: p^2 \nmid u_{\gamma(p)}] \Leftrightarrow [m^2 \mid u_{\gamma(m)} \Rightarrow m \in \{6,12\}].
\end{displaymath}   
\end{abstract}

\vspace*{-\baselineskip}
\maketitle

\setcounter{section}{-1}

\section{Preliminary considerations}

By $\N$, we denote the set of natural numbers including zero and by $\mathbb{P}$ the set of all primes. Let $(u_n)_{n \in \N}$ be the Fibonacci sequence with $u_{0} := 0$, $u_{1} := 1$ and ${u_{n} := u_{n-1}+u_{n-2}}$ for $n \geq 2$ and let ${u_{-1}:=u_1-u_0=1}$. This gives us the following lemma, which can be proven by complete induction.

\begin{lemma} 
\label{0.1}
{\rm{(cf. \cite[2.1 (i)]{Behrends})}}
Let $P:=\begin{pmatrix} 0 & 1\\ 1 & 1 \end{pmatrix}$. Then $P^{n}= \begin{pmatrix} u_{n-1} & u_{n}\\ u_{n} & u_{n+1} \end{pmatrix}$ for $n \in \N$.
\end{lemma}

We write $\restring{m}$ for the set $\{0,1,...,m-1\}$ provided with the usual addition and multiplication modulo $m$. We note that $\restring{m}$ is a ring and for prime numbers even a field (cf. \cite[S.1]{Behrends}). \\
We call $\gamma(m)$ the period of $m$ if $\gamma(m)$ is the smallest positive number $l$ with $u_l \equiv 0 \ (m)$ and $u_{l+1} \equiv 1 \ (m)$ or equivalent: $P^{l}$ is the identity matrix over the ring $\restring{m}$ (cf. \cite[S.4]{Behrends}). Thus, $\gamma(m)$ is the order of $P$ considered as an element in the finite group ${\rm{GL}}_2(\restring{m})$ of the invertible $(2 \times 2)$-matrices with entries $\text{in} \ \restring{m}$. 

\begin{thm}
\label{0.2}
{\rm{(cf. \cite[S.527]{Wall})}}
Let $\epsilon:={\rm{max}}\{a \in \N: p^a \mid u_{\gamma(p)}\}$ and $p \geq 3$ a prime number. Then 
\begin{displaymath}
\gamma(p^k)=\gamma(p) \ \text{for} \ k \in \{1,...,\epsilon\} \ \text{and} \ \gamma(p^{\epsilon+l})=p^{l} \cdot \gamma(p) \ \text{for} \ l \in \N.
\end{displaymath}
\end{thm}

We also remember $2 \mid \gamma(p) \mid \gamma(p^e)$ (cf. \cite[S.4]{Behrends}, Theorem \ref{0.2}). In order to determine the period of natural numbers, the following theorem is helpful, where $\text{lcm}(a_1,...,a_r)$ is abbreviated with $\text{lcm}[a_i]_1^r$.

\begin{thm}
\label{0.3}
{\rm{(cf. \cite[S.526]{Wall})}}
Let $\prod_{i=1}^r p_i^{e_i}$ be the (except for the order of the factors) unique prime factorization of $m \geq 2$. Then $\gamma(m)={\rm{lcm}}[\gamma(p_i^{e_i})]_1^r$.
\end{thm}

Let $\upsilon(m)$ be the number of zeros in $(u_i \ {\rm{mod}} \ m)_{i \in \{0,...,\gamma(m)-1\}}$ (cf. \cite[S.6]{Behrends}). Furthermore, we denote by $\alpha(m)$ the smallest positive number $z$ with $u_z \equiv 0 \ (m)$ (cf. \cite[C(2+3)]{Renault}). Therefore, lemma \ref{0.4} is understandable, whereby here only (v) is proven.

\begin{lemma}
\label{0.4}
The following holds:
\begin{enumerate}
\item
$u_{2n}=u_n \cdot (u_{n-1}+u_{n+1})$ {\rm{(cf. \cite[S.9]{Lausch})}}.
\item
$u_n \equiv 0 \ (p) \Rightarrow u_{n \pm 1} \not \equiv 0 \ (p)$ {\rm{(cf. \cite[3.26]{Lausch})}}.
\item
$\gamma(m)=\upsilon(m) \cdot \alpha(m)$ {\rm{(cf. \cite[C(8)]{Renault})}}.
\item
$\upsilon(m) \in \{1,2,4\}$ {\rm{(cf. \cite[C(9)]{Renault})}}.
\item
$\upsilon(p^e)=\upsilon(p)$ for $p \in \mathbb{P}\setminus\{2\}$ {\rm{(cf. \cite[E(1)]{Renault})}}.
\item
$u_{m-n}=(-1)^n \cdot (u_m \cdot u_{n+1}-u_{m+1} \cdot u_n)$ for $m \geq n$ {\rm{(cf. \cite[1.11]{Lausch})}}.
\end{enumerate}
\end{lemma}

\begin{proof} 
(of (v))
From \glqq $p \mid u_k \Rightarrow \alpha(p) \mid k$\grqq \ follows $\alpha(p) \mid \alpha(p^e) \underset{(iii)}{\Leftrightarrow} \upsilon(p^e) \mid \frac{\gamma(p^e)}{\gamma(p)}\upsilon(p)\underset{\ref{0.2}}{=} p^r\upsilon(p)$ for the right $r \in \N$ and due to $\text{gcd}(\upsilon(p^e),p^r)\underset{(iv)}{=}1$ even $\upsilon(p^e) \mid \upsilon(p)$ (or $\upsilon(p^e) \leq \upsilon(p)$). Conversely is ${\upsilon(p^e) \geq \upsilon(p)}$, which follows for $\upsilon(p)=1$ from (iv), for $\upsilon(p)=2$ from $p^e \mid u_{\gamma(p^e)}\underset{(i)}{=}u_{\frac{\gamma(p^e)}{2}} \cdot (u_{\frac{\gamma(p^e)}{2}-1}+u_{\frac{\gamma(p^e)}{2}+1})$ together with $u_{\frac{\gamma(p^e)}{2}-1}+u_{\frac{\gamma(p^e)}{2}+1} \equiv 2u_{\frac{\gamma(p^e)}{2}-1} \underset{(ii)}{\not \equiv} 0 \ (p)$ and for $\upsilon(p)=4$ similar to $\upsilon(p)=2$.
\end{proof}

\section{Good numbers}

In this section, let $p \neq 2$ be a prime number and $\prod_{i=1}^r p_i^{e_i}$ the (except for the order of the factors) unique prime factorization of $m \geq 2$. We call $m \in \N$ a good number if $P^{\frac{\gamma(m)}{2}}=-{\rm{Id}}$ in $\restring{m}$. Furthermore, $m$ is a prime number, we call $m$ a good prime (cf. \cite[S.5]{Behrends}). We note that $m=2^k \cdot n$ with $k \geq 1$ and odd $n$ is never good because otherwise $P^{\frac{\gamma(m)}{2}}=-\text{Id}$ in $\restring{2^k}$, whereby the goodness of $2^k$ follows in contradiction to $\frac{\gamma(2)}{2} \not \in \N$, $u_{\frac{\gamma(2^2)}{2}} \not \equiv 0 \ (4)$ and $u_{\frac{\gamma(2^k)}{2}-1}\underset{\ref{0.2}}{=}u_{\gamma(2^{k-1})-1} \equiv 1 \ (2^{k-1})$ for $k \geq 3$. In the following, therefore $m$ is odd.

\begin{thm} 
\label{1.1}
$p^e$ is a good number if and only if $p$ is a good prime.
\end{thm}

\begin{proof}
We show \glqq $p^e \ \text{is a good number} \Leftrightarrow \upsilon(p^e) \in \{2,4\}$\grqq, from which the assertion results by lemma \ref{0.4} (v).
From $\upsilon(p^e) \not \in \{2,4\}$ and lemma \ref{0.4} (iv), follows $\upsilon(p^e)=1$ and thus ${u_{\frac{\gamma(p^e)}{2}} \not \equiv 0 \ (p^e)}$ (this shows \glqq $\Rightarrow$\grqq). The other direction follows from $\text{Id}=\left(P^{\frac{\gamma(p^e)}{2}}\right)^2=(r \cdot \text{Id})^2$ in $\restring{p^e}$ which implies $r^2 \equiv 1 \ (p^e)$ and thus $r \equiv -1 \ (p^e)$ due to $P^{\frac{\gamma(p^e)}{2}} \neq \text{Id}$ in $\restring{p^e}$.
\end{proof}

Thus, we can show the main theorem of this chapter, which provides a criterion for classifying odd natural numbers into good and not good numbers. 

\begin{thm} 
\label{1.2}
Let $\gamma(p_i)=2^{k_i} \cdot l_i$ with odd $l_i$. Then $m$ is a good number if and only if all prime numbers $p_i$ are good and $k_1=k_2=...=k_r$.
\end{thm}

\begin{proof}
If $m$ is good, then $p_i \mid m \mid u_{\frac{\gamma(m)}{2}}$ and $p_i \mid m \mid u_{\frac{\gamma(m)}{2} \pm 1}+1$, whereby $P^{\frac{\gamma(m)}{2}}=-\text{Id}$ in $\restring{p_i}$ follows. On the other hand, from the goodness of all $p_i$ follows the goodness of all $p_i^{e_i}$ (cf. theorem \ref{1.1}) and thus $m=\text{lcm}[p_i^{e_i}]_1^r \mid u_{\frac{\text{lcm}[\gamma(p_i^{e_i})]_1^r}{2}}\underset{\ref{0.3}}{=} u_{\frac{\gamma(m)}{2}}$. From the Chinese Remainder theorem, it follows that $u_{\frac{\gamma(m)}{2}+1}=-1$ in $\restring{m}$ if and only if $u_{\frac{\gamma(m)}{2}+1}=-1$ in all $\restring{p_i^{e_i}}$. The second statement is if we put $u_{\frac{\gamma(m)}{2}+1}=u_{a_i \cdot \frac{\gamma(p_i^{e_i})}{2}+1}$ and considering ${u_{k \cdot \gamma(p_i^{e_i})+1} \equiv 1 \ (p_i^{e_i})}$, equivalent to the oddness of all $a_i$. Accordingly, we have to show: 
\begin{center}
\glqq All $a_i$ are odd $\Leftrightarrow {k_1=...=k_r}$\grqq .
\end{center} 
For this, let $\gamma(p_i^{e_i})\underset{\ref{0.2}}{=} 2^{k_i} \cdot p^{r_i} \cdot l_i:=2^{k_i} \cdot h_i$ with $h_i$ odd. Suppose w.l.o.g. $k_1<\text{max}\{k_2,...,k_r\}$. Then $a_1=2^{\text{max}\{k_2,...,k_r\}-k_1} \cdot \frac{\text{lcm}[h_z]_1^r}{h_1}$ is even (this shows \glqq $\Rightarrow$\grqq). On the other hand, it follows from $k_1=...=k_r$ that $a_i=\frac{\text{lcm}[h_z]_1^r}{h_i}$ is odd (this shows \glqq $\Leftarrow$\grqq).
\end{proof}

The theorem proven above allows links between the $\upsilon(p_i)$ and the classification of a natural number as a good number.
Therefore, we need the following lemma, where $2^n \mid\mid \gamma(p)$ means that $2^n \mid \gamma(p)$ and $k \leq n$ holds for all $k \in \N$ with $2^k \mid \gamma(p)$.

\newpage

\begin{lemma} \hspace{0.1mm}
\label{1.3}
\begin{enumerate}
\item
$p$ is a good prime if and only if $\gamma(p) \equiv 0 \ (4)$ {\rm{(cf. \cite[2.4 (i)]{Behrends})}}.
\item
From $\upsilon(p)=1$ follows $2 \mid\mid \gamma(p)$.
\item
From $\upsilon(p)=2$ follows $2^3 \mid \gamma(p)$.
\item
From $\upsilon(p)=4$ follows $2^2 \mid\mid \gamma(p)$.
\end{enumerate}
\end{lemma}

\begin{proof} \hspace{0.1mm}
\begin{enumerate}
\item[(ii)]
The proof of theorem \ref{1.1} shows that only primes $p$ with $\upsilon(p)=1$ are not good. Since $\gamma(p)$ is even, the assertion follows from (i).
\item[(iii)]
With $u_{\frac{\gamma(p)}{2}} \equiv 0 \ (p)$, $u_{\frac{\gamma(p)}{2}-1} \equiv -1 \ (p)$ and $u_{i-1}=u_{i+1}-u_{i}$, we can prove by complete induction $u_{\frac{\gamma(p)}{2}-k} \equiv (-1)^k \cdot u_k \ (p)$, which leads for odd $k=\frac{\gamma(p)}{4} \in \N$ (cf. (i)) to ${u_{\frac{\gamma(p)}{4}} \equiv -u_{\frac{\gamma(p)}{4}} \ (p) \Leftrightarrow p \mid u_{\frac{\gamma(p)}{4}}}$ in contradiction to $\upsilon(p)=2$. Therefore, $\frac{\gamma(p)}{4}$ is even.
\item[(iv)]
Suppose that $\gamma:=\frac{\gamma(p)}{4}$ is even. Subsequently, from $u_\gamma \equiv 0 \ (p)$ and $u_{2\gamma+1} \equiv -1 \ (p)$ follows ${u_{\gamma+1}\underset{\ref{0.4} (vi)}=(-1)^{\gamma} \cdot (u_{2\gamma+1}u_{\gamma+1}-u_{2\gamma+2}u_{\gamma}) \equiv -u_{\gamma+1} \ (p)}$ and thus $u_{\gamma+1} \equiv 0 \ (p)$
in contradiction to lemma \ref{0.4} (ii). \qedhere
\end{enumerate}
\end{proof}

From theorem \ref{1.2} und lemma \ref{1.3}, it follows directly:

\begin{corollary} \hspace{0.1mm}
\label{1.4}
\begin{enumerate}
\item
$m$ is a good number if and only if $k_1=k_2=...=k_r \geq 2$ with $2^{k_i} \mid \mid \gamma(p_i)$.
\item
If $m$ is a good number, then $\upsilon(p_i)=\upsilon(p_1)$ for all $i \in \{1,...,r\}$.
\item
If $\upsilon(p_i)=4$ for all $i \in \{1,...,r\}$, then $m$ is a good number.
\end{enumerate}
\end{corollary}

At the end of this section, we want to show $\upsilon(m)=\upsilon(p_1)$ for good numbers $m$. For this purpose, we need the following theorem, which makes a statement about $\upsilon(m)$ of odd natural numbers $m$:

\begin{thm}
\label{1.5}
\begin{displaymath}
\upsilon(m)=\begin{cases} \upsilon(p_1) & \text{if $\upsilon(p_i)=\upsilon(p_1)$ for all $i \in \{1,...,r\}$} \\ 2 & \text{else} \end{cases}.
\end{displaymath}
\end{thm}

\begin{proof}
We have $u_k \equiv 0 \ (m) \Leftrightarrow [\forall i=1,...,r: \frac{\gamma(p_i^{e_i})}{\upsilon(p_i^{e_i})}=\alpha(p_i^{e_i}) \mid k] \Leftrightarrow \text{lcm}\left[\frac{\gamma(p_i^{e_i})}{\upsilon(p_i^{e_i})}\right]_1^r \mid k$, which implies with $t:=\text{lcm}\left[\frac{\gamma(p_i^{e_i})}{\upsilon(p_i^{e_i})}\right]_1^r$ and $\frac{\gamma(p_i^{e_i})}{\upsilon(p_i^{e_i})} \mid \gamma(p_i^{e_i}) \mid \gamma(m)$
\begin{displaymath}
\upsilon(m)=\left|\left\{d \cdot t \ | \ d \in \N, \ 0 \leq d < \frac{\gamma(m)}{t}\right\}\right|=\frac{\gamma(m)}{t}.
\end{displaymath}
Let $\gamma(p_i^{e_i})=2^{k_i} \cdot h_i$ with odd $h_i$. Together with theorem \ref{0.3} and lemma \ref{0.4} (v), we obtain
\begin{displaymath}
\upsilon(m)=\frac{\gamma(m)}{t}=\frac{\text{lcm}[2^{k_i}]_1^r}{\text{lcm}\left[\frac{2^{k_i}}{\upsilon(p_i)}\right]_1^r}, 
\end{displaymath}
where $\frac{2^{k_i}}{\upsilon(p_i)} \in \N$ holds due to $\upsilon(p_i) \in \{1,2,4\}$ (cf. Lemma \ref{0.4} (iv)). From this, the assertion follows by lemma 1.3 (ii)-(iv) and a study of the cases \glqq ${[\forall i: \upsilon(p_i) \neq 2] \wedge [\exists l_1,l_2: \upsilon(p_{l_1})=1 \wedge \upsilon(p_{l_2})=4]}$\grqq, \glqq ${\exists l: \upsilon(p_l)=2}$\grqq \ and \glqq $\forall i: {\upsilon(p_i)=\upsilon(p_1)}$\grqq.
\end{proof}

The statement \glqq $\upsilon(m)=\upsilon(p_1)$ for good numbers $m$\grqq \ now follows directly from corollary \ref{1.4} (ii).

\section{Wall-Sun-Sun primes}

We call a prime number $p$ a Wall-Sun-Sun prime if $p^2 \mid u_{p-\left(\frac{p}{5}\right)}$ or equivalently $p^2\mid u_{\gamma(p)}$, where $\left(\frac{p}{5}\right)$ is the Legendre symbol of $5$ and $p$. So far (as of November 2015), no such primes have been found (cf. \cite{Weisstein}). In this section, we want to show that the non-existence of Wall-Sun-Sun primes is equivalent to the statement that for $m \in \N\setminus\{0,1\}$ the property \glqq $m^2 \mid u_{\gamma(m)}$\grqq \ is only true for $m=6$ and $m=12$. Only the direction \glqq $\Rightarrow$\grqq \ is not trivial. Therefore, let $m=2^k \cdot n \geq 2$ with $k,n \in \N$ and odd $n$, with which the cases \glqq $k=0$ and $n \geq 3$\grqq \ (cf. lemma \ref{2.1} und \ref{2.2}), \glqq $n=1$\grqq \ (cf. lemma \ref{2.3}) and \glqq $k \geq 1$ and $n \geq 3$\grqq \ (cf. theorem \ref{2.4}) must be studied.

\begin{lemma} \hspace{0.1mm}
\label{2.1}
\begin{enumerate}
\item
$\forall k,n \in \N\setminus\{0\}: u_{k \cdot n}=\sum^{n}_{i=1} \binom n i \cdot u_{i} \cdot u_k^i \cdot u_{k-1}^{n-i}$ {\rm{(cf. \cite[S.18]{Reinecke})}}.
\item
$\forall k,n \in \N \setminus \{0\} \ \forall a,l \in \N: (l \leq k \wedge n^l \mid \frac{u_{a \cdot \gamma(n^k)}}{u_{\gamma(n^k)}} \Rightarrow n^l \mid a)$.
\item
$\forall p \geq 3 \ \forall k \in \N \setminus \{0\}: (p^2 \nmid u_{\gamma(p)} \Rightarrow p^{2k} \nmid u_{\gamma(p^k)})$.
\end{enumerate}
\end{lemma}

\begin{proof} \hspace{0.1mm}
\begin{enumerate}
\item[(ii)]
The cases $a=0$, $a=1$ and $n=1$ are trivial. Let $a,n \geq 2$. By (i) and $n^l \mid n^k \mid u_{\gamma(n^k)}$ follows ${\frac{u_{a \cdot \gamma(n^k)}}{u_{\gamma(n^k)}}=\sum^{a}_{i=1} \binom a i \cdot u_{i} \cdot u_{\gamma(n^k)}^{i-1} \cdot u_{\gamma(n^k)-1}^{a-i}} \equiv a \cdot u_{\gamma(n^k)-1}^{a-1} \ (n^l)$, whereby with $n \nmid u_{\gamma(n^k)-1}^{a-1}$ the assertion follows.
\item[(iii)]
From $p^2 \nmid u_{\gamma(p)}$, $\gamma(p^k)\underset{\ref{0.2}}{=}p^{k-1} \cdot \gamma(p)$ and $p^2 \nmid \frac{u_{p^i \cdot \gamma(p)}}{u_{p^{i-1} \cdot \gamma(p)}}$ (cf. \cite[3.37 (b)]{Lausch}) follows 
\begin{displaymath}
p^{k+1} \nmid u_{\gamma(p^k)}=u_{\gamma(p)} \cdot \prod_{i=1}^{k-1} \frac{u_{p^i \cdot \gamma(p)}}{u_{p^{i-1} \cdot \gamma(p)}}
\end{displaymath} 
and thereby $p^{2k} \nmid u_{\gamma(p^k)}$. \qedhere
\end{enumerate}
\end{proof}

Now we are able to prove lemma \ref{2.2}. Together with lemma \ref{2.3}, it will emerge that $m$ must be an even number, which is not represented by a power of two, to satisfy $m^2 \mid u_{\gamma(m)}$ based on the condition that no Wall-Sun-Sun primes exists.

\begin{lemma}
\label{2.2} 
Let $m=\prod_{i=1}^{r} p_i^{e_i}$ with $e_i \geq 1$ and pairwise different primes $p_i \geq 3$. Then follows from the non-existence of Wall-Sun-Sun primes that $m^2 \nmid u_{\gamma(m)}$.
\end{lemma}

\begin{proof}
Suppose that $m^2=\prod_{i=1}^r p_i^{2e_i} \mid u_{\gamma(m)}$. Since ${p_i^{2e_i} \mid u_{\gamma(m)}}$ and $p_i^{2e_i} \nmid u_{\gamma(p_i^{e_i})}$ (cf. lemma \ref{2.1} (iii)), $\frac{u_{\gamma(m)}}{u_{\gamma(p_i^{e_i})}} \in \N$ is divided by $p_i$. From this together with $\gamma(m)\underset{\ref{0.3}}{=}\frac{\text{lcm}[\gamma(p_j^{e_j})]_1^r}{\gamma(p_i^{e_i})} \cdot \gamma(p_i^{e_i})$ and lemma \ref{2.1} (ii) follows $p_i \mid \frac{\text{lcm}[\gamma(p_j^{e_j})]_1^r}{\gamma(p_i^{e_i})} \mid \frac{\prod_{j=1}^r \gamma(p_j^{e_j})}{\gamma(p_i^{e_i})}=\prod_{\substack{1 \leq j \leq r \\ j \neq i}} \gamma(p_j^{e_j})\underset{\ref{0.2}}{=}\prod_{\substack{1 \leq j \leq r \\ j \neq i}} p_j^{e_j-1}\gamma(p_j)$ and thereby from the divisiveness of the prime factors ${p_i \mid \prod_{\substack{1 \leq j \leq r \\ j \neq i}} \gamma(p_j) \mid \prod_{j=1}^r \gamma(p_j)}$ for all $1 \leq i \leq r$, which implies $\prod_{i=1}^r p_i \mid \prod_{i=1}^r \gamma(p_i)$. \\
It is $\gamma(5)=20$ and $5 \neq p \in P_1 \dot\cup P_2$ with $P_j:=\{p: \gamma(p) \mid j(p+(-1)^j)\}$ (vgl. \cite[2.3]{Behrends}) and thus $\prod\limits_{i=1}^r p_i \mid \prod\limits_{p_i=5} 20 \cdot \prod\limits_{p_i \in P_1} (p_i-1) \cdot \prod\limits_{p_i \in P_2} 2(p_i+1)$. The oddness of all $p_i$ implies $\prod\limits_{i=1}^r p_i \mid \prod\limits_{p_i=5} 5 \cdot \prod\limits_{p_i \in P_1} \frac{p_i-1}{2} \cdot \prod\limits_{p_i \in P_2} \frac{p_i+1}{2}$, whereby $m=5$ in contradiction to $5^2 \nmid u_{\gamma(5)}$ follows.
\end{proof}

\begin{lemma}
\label{2.3}
$(2^k)^2 \nmid u_{\gamma(2^k)}$ holds for all $k \geq 1$.
\end{lemma}

\begin{proof}
The case $k=1$ is immediately apparent. For $k \geq 2$ holds
\begin{displaymath}
{u_{\gamma(2^k)}\underset{\ref{0.2}}{=}u_{2^{k-1} \cdot \gamma(2)}=u_{2 \cdot \gamma(2)} \cdot \prod_{i=2}^{k-1} \frac{u_{2^i \cdot \gamma(2)}}{u_{2^{i-1} \cdot \gamma(2)}}=2^3 \cdot \prod_{i=2}^{k-1} \frac{u_{2^i \cdot \gamma(2)}}{u_{2^{i-1} \cdot \gamma(2)}}}
\end{displaymath}
with $\frac{u_{2^i \cdot \gamma(2)}}{u_{2^{i-1} \cdot \gamma(2)}}\underset{\ref{0.4} (i)}{=} u_{\gamma(2^i)-1}+u_{\gamma(2^i)+1} \equiv 2 \ (4)$ for $i \geq 2$ since $4 \mid 2^i \mid u_{\gamma(2^i) \pm 1}-1$.
Therefore, $2^{k+2} \nmid u_{\gamma(2^k)}$ and due to $k \geq 3$ also $2^{2k} \nmid u_{\gamma(2^k)}$.
\end{proof}

In the proof we showed that $2^{k+1} \mid u_{\gamma(2^k)}$ and $2^{k+2} \nmid u_{\gamma(2^k)}$ for $k \geq 2$, which leads together with $\gamma(2^k)=2^{k-1} \cdot \gamma(2)$ to $2=\upsilon(2^e) \neq \upsilon(2)=1$ for $e \geq 3$. Therefore,  lemma \ref{0.4} (v) holds only for primes $p \neq 2$. The conclusion of this paper is the proof of the following theorem.

\begin{thm}
\label{2.4}
\begin{displaymath}
p^2 \nmid u_{\gamma(p)} \ \text{for all} \ p \in \mathbb{P} \Leftrightarrow [m \in \N \setminus \{0,1\} \wedge m^2 \mid u_{\gamma(m)} \Rightarrow m \in \{6,12\}].
\end{displaymath}
\end{thm}

\begin{proof}
The direction \glqq $\Leftarrow$\grqq \ is trivial. Due to lemma \ref{2.2} and \ref{2.3}, we only need to examine numbers $m=2^k \cdot n$ with $k \geq 1$ and odd $n \geq 3$. \\ 
Due to $\text{gcd}(2^k,n)=1$, we obtain by theorem \ref{0.3} $\gamma(m)=\text{lcm}(\gamma(2^k),\gamma(n))$. In the following, let ${h_1:=n}$ and $h_2:=2^k$. Suppose $m^2=h_1^2 \cdot h_2^2 \mid u_{\gamma(m)}$. Then ${h_j^2 \mid u_{\gamma(m)}}$, whereby together with $u_{\gamma(h_j)} \mid u_{\gamma(m)}$ and $j^2 \cdot h_j^{|(-1)^j j-1|} \nmid u_{\gamma(h_j)}$ (cf. for $h_1$ lemma \ref{2.2} and for $h_2$ the proof of lemma \ref{2.3}) $\frac{h_j}{j} \mid \frac{u_{\gamma(m)}}{u_{\gamma(h_j)}}$ follows. With this and $\gamma(m)=\frac{\text{lcm}(\gamma(h_1),\gamma(h_2))}{\gamma(h_j)} \cdot \gamma(h_j)$, we conclude from lemma \ref{2.1} (ii) $\frac{h_j}{j} \mid \frac{\text{lcm}(\gamma(h_1),\gamma(h_2))}{\gamma(h_j)}$. Thereby, $n=3$, because from ${n \mid \frac{\gamma(n) \cdot \gamma(2^k)}{\gamma(n)}=2^{k-1} \cdot \gamma(2)}$ and $\text{gcd}(2^{k-1},n)=1$ follows $n \mid \gamma(2)=3$, and ${k \in \{1,2\}}$ because $2^{k-1} \nmid \frac{\text{lcm}(\gamma(3),\gamma(2^k))}{\gamma(2^k)}=1$ for $k \geq 4$ and $2^{3-1} \nmid \frac{\text{lcm}(\gamma(3),\gamma(2^3))}{\gamma(2^3)}=2$. Calculating shows ${6^2 \mid 12^2 \mid u_{\gamma(6)}=u_{\gamma(12)}=u_{24}=46368}$.
\end{proof}

\end{document}